\documentclass[12pt]{article}
\usepackage{graphics,amsmath,amssymb}
\usepackage{amsthm}
\usepackage{amsfonts}
\usepackage{latexsym}
\usepackage{amscd}

\newtheorem{theorem}{Theorem}

\newtheorem{proposition}[theorem]{Proposition}

\begin{document}

\begin{center}
\title{A uniformly convergent Series for $\zeta(s)$ and closed Formulas, that include Catalan Numbers} 
\author{Robert J. Betts}
\maketitle
\emph{Graduate Department of Mathematics, University of Massachusetts Lowell, One University Avenue, Lowell, Massachusetts 01854 Robert\_Betts@alum.umb.edu}
\end{center}
\begin{abstract}
There exists an infinite series of ratios by which one can derive the Riemann zeta function $\zeta(s)$ from Catalan numbers and central binomial coefficients which appear in the terms of the series. While admittedly the derivation is not deep it does indicate some combinatorial aspect to the Riemann zeta function. But we actually do find also four additional new closed formulas, which include a formula by which one can compute $\zeta(s)$ for a countably infinite number of discrete positive values for $s$ where the formula contains Catalan numbers not in infinite series. The Riemann zeta function has applications in physics, such as in computations related to the Casimir effect. Our result indicates a link between the Riemann zeta function, combinatorics, Catalan numbers, the central binomial coefficient and the content of a hypersphere, under certain conditions. 
\end{abstract}

\section{Catalan Numbers and $\zeta(s)$}
There exists a combinatorial relationship between Catalan numbers, the central binomial coefficient and the Riemann zeta function. By using Catalan numbers $C_{n}$ and central binomial coefficients ${2n \choose n}$, A combinatorial substitution derives $\zeta(s)$ as a uniformly convergent infinite series of analytic functions on $\mathbb{C}$ for \(|s| > 1\) that contains $C_{n}$ and ${2n \choose n}$ (Proposition~\ref{prop1}, Section 3, this paper). Each Catalan number also can be expressed in terms of $\zeta(s)$ (Proposition 2, Section 3) but also in closed form (See Eq. (21), Section 3.1). This shows the actual relevance Catalan numbers and the central binomial coefficient have in other areas besides combinatorial number theory. We also find a connection between Catalan numbers, the Riemann zeta function and the content of a hypersphere in even-numbered Euclidean spaces of dimensions divisible by four (See Section 3). 

\indent Important applications in mathematical physics have been found for $\zeta(s)$, for one example as a pseudodifferential operator
\begin{eqnarray}
s \rightarrow &\frac{1}{2} + i\Box,\nonumber \\
\zeta(s) \rightarrow &\zeta(\frac{1}{2} + i\Box),
\end{eqnarray}
(where ``$\Box$" is the d'Alembertian operator). It is used to investigate various cosmological models in string theory~\cite{Arefeva} (page 2),~\cite{Fulman},~\cite{Kowalski},~\cite{Peterson},~\cite{Schwabl},~\cite{Terras}. 
\section{Catalan numbers and their generating function}
The \emph{Catalan numbers} for positive integers $n$,
\begin{equation}
C_{n} = \frac{{2n \choose n}}{n + 1}
\end{equation}
appear in integer sequence number A000108 in N. J. A. Sloane's Online Encyclopedia of Integer Sequences~\cite{Sloane}. As an infinite sequence of positive integers they can be represented as 
$$
\left\{\frac{{2n \choose n}}{n + 1}\right\}_{n = 1}^{\infty}.
$$
The generating function for Catalan numbers is
\begin{equation}
\frac{1 - \sqrt{1 - 4x}}{2x} = \sum_{n = 0}^{\infty} \! C_{n}x^{n}.
\end{equation}
On the complex plane $\mathbb{C}$, $\Gamma(z)$ denotes the Gamma function. For all complex $z$ such that \(\mathcal{R}e(z) > 0\),
$$
\Gamma(z) = \int_{0}^{\infty}t^{z - 1}e^{-t}dt,
$$
$$
\Gamma(2z) = \int_{0}^{\infty}t^{2z - 1}e^{-t}dt.
$$
For \(\mathcal{R}e(z) > 0\) the function $\Gamma(z)$ is analytic inside all domains on $\mathbb{C}$ where \(|z| \gg 1\). There is an analytic continuation
\begin{eqnarray}
 &\Gamma(z + 1) = z\Gamma(z),\nonumber \\
 &\Gamma(2z + 1) = 2z\Gamma(2z),\nonumber
\end{eqnarray}
as \(|z| \rightarrow \infty\). So with \(\mathcal{R}e(z) > 0\) and with $|z|$ large, one way to compute an asymptotic approximation to 
$$
C_{n} = \frac{{2n \choose n}}{n + 1}
$$
would be to start from first principles, meaning by the brute force computation of~\cite{Dettman} (Compare Section 4.10, pages 191--192) 
$$
\frac{\Gamma(2z + 1)}{(z + 1)[\Gamma(z + 1)]^{2}}
$$
$$
= \frac{\int_{0}^{\infty}t^{2z}e^{-t}dt}{(z + 1)(\int_{0}^{\infty}t^{z}e^{-t}dt)^{2}}
$$
in all domains $D$ where \(\mathcal{R}e(z) > 0\) and where $|z|$ is large, letting \(t = 2sz\) for variable $s$, expand both the numerator and the denominator (~\cite{Dettman}, Example 10.3.1), after which one divides the expanded numerator by the expanded denominator, a formidable computation to say the least! However the various derivations for Stirling's Formula in the literature do all the necessary work for us. 
\section{Deriving $\zeta(s)$ through the Use of $C_{n}$, ${2n \choose n}$ in a Series of complex--valued Ratios}
In this Section we find an infinite series containing Catalan numbers through which one can derive the values for the Riemann zeta function $\zeta(s)$. 

\noindent
\begin{proposition} 
Let \(s \in \mathbb{C}\). Then
\begin{eqnarray}
\zeta(s) & = & \sum_{n = 1}^{\infty}\frac{1}{n^{s}} \\
         & = & \left(\frac{C_{0}}{1}\right)^{s} + \sum_{n = 1}^{\infty}\left(\frac{C_{n}}{{2n \choose n}}\right)^{s},
\end{eqnarray}
\emph{where we define \(C_{0} = 1\) as the Catalan number for \(n = 0\).} \\
\label{prop1}
\end{proposition}

\begin{proof}
Using the definition for Catalan numbers we derive
\begin{equation}
\frac{{2n \choose n}}{n + 1} = C_{n} 
\end{equation}
$$
\Longrightarrow \frac{1}{n + 1} = \frac{C_{n}}{{2n \choose n}}.
$$
Then 
\begin{eqnarray}
\zeta(s) & = & \sum_{n = 1}^{\infty}\frac{1}{n^{s}} = \frac{1}{1^{s}} + \sum_{n = 1}^{\infty}\frac{1}{(n + 1)^{s}} \\
         & \equiv & \left(\frac{C_{0}}{1}\right)^{s} + \sum_{n = 1}^{\infty}\left(\frac{C_{n}}{{2n \choose n}}\right)^{s},
\end{eqnarray}
where \(C_{0} = \frac{{2(0) \choose 0}}{0 + 1} = 1\) when \(n = 0\).\\ 
\end{proof}

\indent Admittedly  Equations (7)--(8) are not closed form expressions. This is derived independently of any intrinsic property of Catalan numbers. This however does not mean that suitable closed formulas do not exist. (See Section 3.1). Due to this one could feel inclined to argue that to substitute the quotient
$$
\frac{C_{n}}{{2n \choose n}},
$$
for
$$
\frac{1}{n + 1},
$$
in computations has little practical advantage or consequence. There are three counterarguments to this reasoning. First, the very definition for Catalan numbers is
$$
C_{n} = \frac{{2n \choose n}}{n + 1}.
$$
Certainly very few would contend that to perform computations with the ratio on the right hand side is inconsequential, since this is just the $n^{th}$ Catalan number. Otherwise how could one know how to compute the $n^{th}$ Catalan number if either tables of Catalan numbers or else computer access to the OEIS is not readily available to the researcher? Second, although it certainly is true that in the field $\mathbb{Q}$,
$$
\frac{C_{n}}{{2n \choose n}} = \frac{1}{n + 1},
$$
we still have that \(C_{n} \not = 1\) and \({2n \choose n} \not = n + 1\). Also and for mere sake of comparison, the expression for Wallis's formula~\cite{Abramowitz} (See Eq. 6.1.49, Chapter 6, page 258) does not appear at all in lowest terms anymore than does our ratio $\frac{C_{n}}{{2n \choose n}}$. Therefore the infinite series of ratios in Eq. (8) is, in its own right, a uniformly convergent series of analytic functions on $\mathbb{C}$ when \(|s| > 1\), where of course for sake of convenience and if we choose to (e. g., for ease of computation) we may substitute $\frac{1}{n + 1}$ for each term $\frac{C_{n}}{{2n \choose n}}$. Third, whereas it is true that
$$
\frac{C_{n}}{{2n \choose n}} = \frac{1}{n + 1}
$$
always for the two rational numbers on both sides of the equal sign, it is not true always that on $\mathbb{C}$,
$$
\left(\frac{C_{n}}{{2n \choose n}}\right)^{s_{1}} = \left(\frac{1}{n + 1}\right)^{s_{2}},\: s_{1}, s_{2} \in \mathbb{C},
$$
particularly even when \(s_{1} = s_{2}\) holds. In fact we have on the complex plane, a result
$$
exp\left(s_{1} \log \left(\frac{C_{n}}{{2n \choose n}}\right) - s_{2} \log \left(\frac{1}{n + 1}\right)\right) = 1,\: s_{1}, s_{2} \in \mathbb{C}
$$
$$
\Leftrightarrow s_{1}\log\left(\frac{C_{n}}{{2n \choose n}}\right) = s_{2}\log\left(\frac{1}{n + 1}\right) + 2m\pi i, 
$$
$$
m = 0, \: \pm 1, \: \pm 2, \ldots
$$
This means if the two analytic functions when expressed on $\mathbb{C}$ as \(exp(s_{1} \log(\frac{C_{n}}{{2n \choose n}}))\) and as \(exp(s_{2} \log\left(\frac{1}{n + 1}\right))\) are not on the same Riemann sheet, they will not always be equal even when \(s_{1} = s_{2}\). Therefore when it is required we can consider for \(s \in \mathbb{C}\), the principal values on both sides of Eq. (7). What is more significant is that Proposition~\ref{prop1} actually indicates the Riemann zeta function \emph{has some combinatorial description} (We show this more forthrightly, independently of Proposition 1, in Section 3.1). It can be derived by an infinite series of ratio terms with Catalan number $C_{n}$ in the numerator of each term, central binomial coefficient ${2n \choose n}$ in the denominator of each term, \(n = 1, 2, 3, \ldots\) and with each ratio term raised to the complex power $s$. Each term then is an element of $\mathbb{C}$. Although one might be inclined to think there is no intrinsic property of Catalan numbers allowing us to derive $\zeta(s)$ from them, we show in the following Subsection that this actually is not the case. We in fact derive as well an unexpected result, namely that there are a countably infinite number of positive integer values of $s$ for which $\zeta(s)$ and Catalan numbers are related to the content $V_{m}$ of a hypersphere of even dimension where this Euclidean dimension is \(m \equiv 0 \: mod \: 4\). There do exist two closed form expressions by which one can derive the Riemann zeta funtion from using Catalan numbers and vice versa for a countably infinite number of values of $s$ (Section 3.1). These closed formulas we shall find without the use of a uniformly convergent infinite series such as in Proposition 1 (See Eq. (7)). Certainly there is practical advantage in using these closed form expressions rather than the result in Eq. (7). Moreover there are several practical consequences with these closed formulas for both pure and applied mathematics, as we will show (See Section 5 and Section 4).

\indent When \(s = 2\) Leonard Euler (~\cite{Ireland}, page 229) proved the result:
\begin{equation}
\zeta(2) = \frac{\pi^{2}}{6}.
\end{equation}
Since it is so straightforward we need not work it out here, but as an exercise the reader actually can convince himself or herself (for example through Euler-MacLaurin summation~\cite{Abramowitz}. See Eq. 23.1.30, page 806) that, when \(s = 2\),
\begin{equation}
\lim_{n \rightarrow \infty} \left(\frac{1}{1^{2}} + \sum_{n = 1}^{\infty}\left(\frac{C_{n}}{{2n \choose n}}\right)^{2}\right) 
\end{equation}
$$
= \frac{\pi^{2}}{6}.
$$
In fact when we approximate with only sixteen terms (the following computations were done using PARI Calculator Version 2.3.1, GNU General Public License), we get the following (Compare the entries in Table 3, when \(s = 1\).):
$$
\frac{1}{1^{2}} + \sum_{n = 1}^{15} \! \left(\frac{C_{n}}{{2n \choose n}}\right)^{2}
$$
\begin{eqnarray}
=  & 1 + \left(\frac{C_{1}}{{2 \choose 1}}\right)^{2} +  \left(\frac{C_{2}}{{4 \choose 2}}\right)^{2}    +  \left(\frac{C_{3}}{{6 \choose 3}}\right)^{2}    +  \left(\frac{C_{4}}{{8 \choose 4}}\right)^{2}    +  \left(\frac{C_{5}}{{10 \choose 5}}\right)^{2} \nonumber \\
+  & \left(\frac{C_{6}}{{12 \choose 6}}\right)^{2}    +  \left(\frac{C_{7}}{{14 \choose 7}}\right)^{2}   +  \left(\frac{C_{8}}{{16 \choose 8}}\right)^{2}   +  \left(\frac{C_{9}}{{18 \choose 9}}\right)^{2}   +  \left(\frac{C_{10}}{{20 \choose 10}}\right)^{2} \nonumber \\
+  & \left(\frac{C_{11}}{{22 \choose 11}}\right)^{2}  +  \left(\frac{C_{12}}{{24 \choose 12}}\right)^{2} +  \left(\frac{C_{13}}{{26 \choose 13}}\right)^{2} +  \left(\frac{C_{14}}{{28 \choose 14}}\right)^{2} +  \left(\frac{C_{15}}{{30 \choose 15}}\right)^{2} \nonumber 
\end{eqnarray}
\begin{equation}
= \sum_{n = 1}^{16}\frac{1}{n^{2}}
\end{equation}
$$
= 1.5843\ldots \approx 1.6449\ldots = \frac{\pi^{2}}{6},
$$
where the Catalan numbers in the numerator for this approximating computation in (10) can be found at The Online Encyclopedia of Integer Sequences~\cite{Sloane} and where the central binomial coefficients in the denominators can be found from Table 24.1, page 828, provided by Abramowitz and Stegun~\cite{Abramowitz}. So with approximating by only sixteen terms we obtain a relative error of only 
$$
\varepsilon = \left|\frac{1.6449\ldots - 1.5843\ldots}{1.6449\ldots}\right| = 0.0368\ldots
$$
for a percent relative error of $3.68\%$.

\indent Now we compare the sixteen term approximation with Euler's product formula (~\cite{Ireland}, Proposition 16.1.1,~\cite{Edwards}, Section 1.2 and Section 1.9), with primes ranging from \(p = 2\) to \(p = 53\):
$$
\zeta(2) \approx \prod_{p = 2}^{p = 53}\left(1 - \frac{1}{p^{2}}\right)^{-1}
$$
\begin{equation}
= 1.6392\ldots
\end{equation}
This value (i. e., for sixteen terms) makes a percent relative error with \(\zeta(2) = \sum_{n = 1}^{\infty}\frac{1}{n^{2}}\) of $0.0035\%$ and a percent relative error for the sixteen terms from the series expansion of $3.349\%$. We can round off the decimal value up to the nearest tenth in the finite sum containing Catalan numbers, to obtain
\begin{eqnarray}
\zeta(2) & =       & \frac{\pi^{2}}{6} \nonumber \\
         & =       & \prod_{p}\left(1 - \frac{1}{p^{2}}\right)^{-1} \nonumber \\
         & \approx & \frac{1}{1^{2}} + \sum_{n = 1}^{15}\left(\frac{C_{n}}{{2n \choose n}}\right)^{2} \nonumber \\
\end{eqnarray}
\begin{equation}
= 1.6.  
\end{equation}  
\indent
Admittedly the infinite sum can have slow convergence for \(|s| > 1\). Nevertheless one can verify through computer programs (in MATLAB sessions or by PARI Calculator, for example) and for large $n$ that although there indeed is slow convergence it is very close to \(\pi^{2}/6 = 1.6449\dots\), when \(s = 2\). The infinite series of analytic functions
$$
\left(\frac{C_{n}}{{2n \choose n}}\right)^{s}
$$
actually converges uniformly on $\mathbb{C}$ by the Weierstrass--M test for some absolutely convergent infinite series of real numbers $M_{n}$, such that
$$
\left|\left(\frac{C_{n}}{{2n \choose n}}\right)^{s}\right| \leq M_{n}.
$$
For example when \(s > 1\) is an integer let
$$
M_{n} = \left(\frac{1}{n}\right)^{2}.
$$
For positive integers $n$ and for each integer \(s > 1\),
\begin{equation}
\left|\left(\frac{C_{n}}{{2n \choose n}}\right)^{s}\right| \leq \frac{1}{n^{2}},
\end{equation}
meaning the infinite series
$$
\sum_{n = 1}^{\infty}\left(\frac{C_{n}}{{2n \choose n}}\right)^{s} = \sum_{n = 1}^{\infty}\left(\frac{1}{n + 1}\right)^{s}
$$
converges uniformly on $\mathbb{C}$ for all integers \(s > 1\).   
\begin{center}
\begin{tabular}{|l                   |c                                              |c                        |c                                                  |r|}
\hline

                  $n$      &          ${2n \choose n}$                  &               $C_{n}$               &    $\frac{C_{n}}{{2n \choose n}}$ $=$      &        $=$ $n+1^{st}$ term of $\zeta(1)$, $n \geq 1$, $s = 1$  \\   
\hline
                  $1$      &          ${2 \choose 1}$   $=$         $2$     &                     $1$         &    $0.5000$  $=$                           &        $=$  $1/2$                                               \\
\hline                 
                  $2$      &          ${4 \choose 2}$   $=$         $6$     &                     $2$         &    $0.3333$  $=$                           &        $=$  $1/3$                                                \\
\hline                  
                  $3$      &          ${6 \choose 3}$   $=$        $20$     &                     $5$         &    $0.2500$  $=$                           &        $=$  $1/4$                                                 \\
\hline                   
                  $4$      &          ${8 \choose 4}$   $=$        $70$     &                    $14$         &    $0.2000$  $=$                           &        $=$  $1/5$                                                  \\
\hline                 
                  $5$      &          ${10 \choose 5}$  $=$       $252$     &                    $42$         &    $0.1666$  $=$                           &        $=$  $1/6$                                                   \\
\hline                   
                  $6$      &          ${12 \choose 6}$  $=$       $924$     &                   $132$         &    $0.1428$  $=$                           &        $=$  $1/7$                                                    \\
\hline                    
                  $7$      &          ${14 \choose 7}$  $=$      $3432$     &                   $429$         &    $0.1250$  $=$                           &        $=$  $1/8$                                                     \\
\hline                   
                  $8$      &          ${16 \choose 8}$  $=$     $12870$     &                  $1430$         &    $0.1111$  $=$                           &        $=$  $1/9$                                                      \\
\hline                    
                  $9$      &          ${18 \choose 9}$  $=$     $48620$     &                  $4862$         &    $0.1000$  $=$                           &        $=$ $1/10$                                                       \\
\hline                    
                  $10$     &          ${20 \choose 10}$ $=$    $184756$     &                 $16796$         &    $0.0909$  $=$                           &        $=$ $1/11$                                                        \\
\hline                   
                  $11$     &          ${22 \choose 11}$ $=$    $705432$     &                 $58786$         &    $0.0833$  $=$                           &        $=$ $1/12$                                                         \\  
\hline                   
                  $12$     &          ${24 \choose 12}$ $=$   $2704156$     &                $208012$         &    $0.0769$  $=$                           &        $=$ $1/13$                                                          \\
\hline                  
                  $13$     &          ${26 \choose 13}$ $=$  $10400600$     &                $742900$         &    $0.0714$  $=$                           &        $=$ $1/14$                                                           \\
\hline                  
                  $14$     &          ${28 \choose 14}$ $=$  $40116600$     &               $2674440$         &    $0.0666$  $=$                           &        $=$ $1/15$                                                            \\
\hline                  
                  $15$     &          ${30 \choose 15}$ $=$ $155117520$     &               $9694845$         &    $0.0625$  $=$                           &        $=$ $1/16$                                                             \\

\hline
\end{tabular}
\end{center}
\begin{center}
Table 3.
\end{center}
We can express Catalan numbers in terms of $\zeta(s)$ and the central binomial coefficient ${2n \choose n}$.

\noindent
\begin{proposition} 
For integers $n \geq 1$,
\begin{equation}
C_{n} = {2n \choose n}\left(\zeta(s) - \sum_{i = 1}^{n}\frac{1}{i^{s}} - \sum_{i = n + 2}^{\infty}\frac{1}{i^{s}}\right)^{1/s}.
\end{equation}
\label{prop2}
\end{proposition}
\begin{proof}
\begin{equation}
\zeta(s) = \frac{1}{1^{s}} + \frac{1}{2^{s}} + \ldots + \frac{1}{n^{s}} + \frac{1}{(n + 1)^{s}} + \ldots
\end{equation}
$$
\Longrightarrow \left(\frac{1}{n + 1}\right)^{s} = \zeta(s) - \sum_{i = 1}^{n}\frac{1}{i^{s}} - \sum_{i = n + 2}^{\infty}\frac{1}{i^{s}}.
$$
Then from the definition of Catalan numbers we have through substitution
\begin{eqnarray}
\left(\frac{C_{n}}{{2n \choose n}}\right)^{s}  & = & \zeta(s) - \sum_{i = 1}^{n}\frac{1}{i^{s}} - \sum_{i = n + 2}^{\infty}\frac{1}{i^{s}} \nonumber \\
\Longrightarrow \frac{C_{n}}{{2n \choose n}}   & = & \left(\zeta(s) - \sum_{i = 1}^{n}\frac{1}{i^{s}} - \sum_{i = n + 2}^{\infty}\frac{1}{i^{s}}\right)^{1/s} \nonumber \\
\Longrightarrow C_{n}                          & = & {2n \choose n}\left(\zeta(s) - \sum_{i = 1}^{n}\frac{1}{i^{s}} - \sum_{i = n + 2}^{\infty}\frac{1}{i^{s}}\right)^{1/s}.
\end{eqnarray}
\end{proof}

\indent We can choose to define
\begin{equation}
C\left(n;\zeta(s);{2n \choose n}\right) := C_{n},
\end{equation}
whenever the $n^{th}$ Catalan number $C_{n}$ is expressed by the result in Proposition~\ref{prop2}.
\subsection{Catalan numbers, even Bernoulli Numbers, $\zeta(2n)$ and the Content of a Hypersphere in even Euclidean Dimension 4n}
In the following Subsection it is to be understood one represents the series expansion for $\zeta(s)$ as
$$
\zeta(s) = \sum_{k = 1}^{\infty}\frac{1}{k^{s}}, |s| > 1,
$$
for \(k = 1, 2, 3\ldots\), when \(s = 2n\), instead of as in previous Sections,
$$
\zeta(s) = \sum_{n = 1}^{\infty}\frac{1}{n^{s}}, |s| > 1.
$$
Here we calculate a closed formula by which one can compute Catalan numbers when \(s = 2n\), \(n = 1, 2, 3\dots\), by using $\zeta(2n)$ and Bernoulli numbers $B_{2n}$, when \(s = 2n\). A formula of significance for this can be found in the well known mathematical tome edited by M. Abramowitz and I. Stegun~\cite{Abramowitz} (See Eq. 23.2.16, Chapter 23, page 807):
\begin{equation}
\zeta(2n) = \frac{2^{2n - 1}\pi^{2n}}{(2n)!}|B_{2n}|.
\end{equation}
\noindent
\begin{proposition}
\emph{Let $C_{n}$ be the $n^{th}$ Catalan number. Then, for \(r_{n} \in \mathbb{R}\),}
\begin{eqnarray}
\zeta(2n)&=&\frac{\pi^{2n}r_{n}^{4n}}{n!(n + 1)!C_{n}},\nonumber \\
C_{n}&=&\frac{\pi^{2n}r_{n}^{4n}}{n!(n + 1)!\zeta(2n)},
\end{eqnarray}
\emph{where}
\begin{equation}
r_{n} = \sqrt[4n]{2^{2n - 1}|B_{2n}|}
\end{equation}
\emph{is the radius of a hypersphere in Euclidean dimension \(m = 4n\) and $B_{2n}$ is the $2n^{th}$ Bernoulli number. What is more, $\zeta(s)$ gives the content $V_{4n}$ of a hypersphere in Euclidean dimension \(m = 4n\) when \(s = m/2\).} \\ \noindent
\label{prop3}
\end{proposition}
\begin{proof}
In Euclidean space of even dimension $m$, the content of a hypersphere of radius $r$ is~\cite{Apostol}
\begin{equation}
V_{m} = \frac{\pi^{m/2}r^{m}}{(m/2)!}.
\end{equation}
then replace $r$ with $r_{n}$, let \(m = 4n\), \(r_{n}^{4n} = 2^{2n - 1}|B_{2n}|\). Then from Eq. (20),
\begin{eqnarray}
\zeta(2n)&=&\frac{2^{2n - 1}\pi^{2n}}{(2n)!}|B_{2n}|\nonumber \\
         &=&\frac{\pi^{2n}r_{n}^{4n}}{(2n)!}\nonumber \\
         &=&\frac{\pi^{2n}r_{n}^{4n}}{n!(n + 1)!C_{n}}.
\end{eqnarray}
Similarly, we can rearrange Eq. (24) to obtain
\begin{equation}
C_{n} = \frac{\pi^{2n}r_{n}^{4n}}{n!(n + 1)!\zeta(2n)}.
\end{equation}
Now to complete the proof one need only compare Eq. (20) to Eq. (23) to see they always obtain the same value when \(s = 2n\), \(m = 4n\) and when
$$
r = r_{n} = \sqrt[4n]{2^{2n - 1}|B_{2n}|},
$$
that is, \(V_{4n} = \zeta(2n)\), for Euclidean dimension \(m = 4n\) for given radius $r_{n}$.
\end{proof}
When \(s = 2n\), \(n = 1, 2, 3\ldots\) there is practical advantage in using Proposition 3 over Proposition 1 to compute $\zeta(2n)$. Certainly this is true even with machine computations. For instance suppose one uses an M--file however short to compute $\zeta(2n)$ in a MATLAB session, using the two infinite series that appear in Eqtns. (4)--(5) in Proposition 1. The two infinite series would have to be summed term by term by machine for large $n$. However if one has no access either to high speed computers or to computer algebra software packages (e.g., MATLAB, MATHEMATICA, PARI or Maxima), one can use a hand calculator to compute $\zeta(2n)$, using the closed formula in Proposition 3 which contains always the $n^{th}$ Catalan number for each $n$. This means in the first closed formula in Eq. (21) one actually can make computations of $\zeta(2n)$ with a hand held calculator (e. g., a Texas Instrument TI--30XA) at least for values of $n$ where \(C_{n} = \frac{{2n \choose n}}{n + 1} = O(10^{9})\).

\indent So we have derived a closed formula to compute $\zeta(2n)$ using Catalan numbers $C_{n}$ (Eq. (24)) and a closed formula for $C_{n}$ using $\zeta(2n)$ (Eq. (25)). What is more we have established that (See Eqtns. (20)--(25))
\begin{eqnarray}
V_{4n}&=&\frac{\pi^{2n}r_{n}^{4n}}{(2n)!}\nonumber \\
      &=&\frac{\pi^{2n}r_{n}^{4n}}{n!(n + 1)!C_{n}} = \zeta(2n).
\end{eqnarray}
That is, in Euclidean space of even dimension \(m = 4n\) and when \(s = 2n\), $\zeta(2n)$ actually is the content of a hypersphere of radius 
$$
r_{n} = \sqrt[4n]{2^{2n - 1}|B_{2n}|}, 
$$
where $B_{2n}$ is the $2n^{th}$ Bernoulli number (Compare Eq. (23)). Results appear in the Table for \(1 \leq n \leq 10\). The author chose \(\pi = 3.14159\ldots\) using the first five figures to the right of the decimal point for the computation of $\pi^{2n}$, \(1 \leq n \leq 10\). The values for the even Bernoulli numbers used for the computed results for the hypersphere radius $r_{n}$ (See Eq. (22)) that appear in Table 2 are found in the written compilation by editors Abramowitz and Stegun~\cite{Abramowitz} (Table 23.2, Chapter 23, page 810). The values for $r_{n}$, $V_{4n}$ and $\zeta(2n)$ were determined from using Eqtns. (22)--(24) and by using PARI GP Calculator in an MSDOS command window along with the PARI functions \textbf{zeta()} and \textbf{factorial()}. In Table 2 when $6 \leq n \leq 10$ the data for $V_{4n}$ are bounded above by one while the data for $\zeta(2n)$ are bounded below by one.

\pagebreak

\begin{center}
\begin{tabular}{|l                 |c                       |c                                                                              |r|}
\hline
                   $n$     &       $r_{n}$           &       $V_{4n}$               &        $\zeta(2n)$        \\   
\hline
                   $1$     &       $0.75983\ldots$   &       $1.64480\ldots$        &        $1.64493\ldots$     \\
 
                   $2$     &       $0.84770\ldots$   &       $1.0822416\ldots$      &        $1.08232\ldots$     \\
                  
                   $3$     &       $0.97759\ldots$   &       $1.0172917\ldots$      &        $1.01734\ldots$      \\ 

                   $4$     &       $1.09491\ldots$   &       $1.0039941\ldots$      &        $1.00407\ldots$      \\
                  
                   $5$     &       $1.20070\ldots$   &       $1.0009727\ldots$      &        $1.00099\ldots$      \\
                  
                   $6$     &       $1.29745\ldots$   &       $0.9991535\ldots$      &        $1.00024\ldots$       \\
 
                   $7$     &       $1.38719\ldots$   &       $0.9987764\ldots$      &         $1.00006\ldots$       \\
                    
                   $8$     &       $1.47127\ldots$   &       $0.9998357\ldots$      &        $1.0000152\ldots$      \\ 

                   $9$     &       $1.55037\ldots$   &       $0.9954579\ldots$      &        $1.0000038\ldots$      \\
 
                  $10$     &       $1.62584\ldots$   &       $0.9999485\ldots$      &        $1.00000095\ldots$    \\

\hline
\end{tabular}
\end{center}
\begin{center}
Table 2.
\end{center}

\indent In Euclidean $4n$ dimensions let the closed disk $D$
$$
x_{1}^{2} + x_{2}^{2} + \cdots + x_{4n}^{2} \leq r_{n}^{2},
$$
$$
(x_{1}, x_{2}, \ldots, x_{4n}) \in \mathbb{R}^{4n} \subseteq \mathbb{R}^{4n + 1},
$$
denote the hypersphere of radius $r_{n}$ and content \(V_{4n} = \zeta(2n)\) (See Proposition 3). This equation for the hypersphere is a quadratic form. In the vector space $\mathbb{R}^{4n}$ it has group actions from symplectic group $Sp(4n, \mathbb{R})$, which also has group actions on
$$
x_{1}^{2} + x_{2}^{2} + \cdots + x_{4n}^{2} \leq 1.
$$
So in $\mathbb{R}^{4n}$ the former hypersphere of radius $r_{n}$ can be derived by group actions on the latter hypersphere of unit radius, and vice versa. So if in fact we allow for some dimensional units \(r_{n}^{2} \cdot 1/r_{n}^{2} = 1\), we obtain an identification between both spheres. The manifold $S^{4n}$ itself is associated with the complex manifold $\mathbb{C}\mathbb{P}^{2n}$ through the Hopf fibration~\cite{Baez},~\cite{Madsen} and where \(S^{2(2n) - 1} \subseteq S^{2(2n)} \subseteq S^{2(2n) + 1}\)~\cite{Madsen} (Chapter 14), for which there is a generalized Hopf fibration map 
$$
S^{2(2n) + 1} \mapsto \mathbb{C}\mathbb{P}^{2n}.
$$
When it is imbedded in Euclidean space $\mathbb{R}^{4n}$ its manifold $\partial(D)$ is homeomorphic to $S^{4n - 1}$ which is a submanifold of $S^{4n}$. What one has is the compactification of $S^{4n - 1}$ in $\mathbb{R}^{4n}$ and the compactification of $S^{4n}$ in $\mathbb{R}^{4n + 1}$, where $S^{4n - 1} \subseteq$ $S^{4n}$ $\subseteq$ $\mathbb{R}^{4n + 1}$. Therefore for the construction of charts and atlases the geometer must consider either the submanifold $S^{4n - 1}$ of the topological manifold $S^{4n}$ or else this manifold $S^{4n}$ in the imbedding space $\mathbb{R}^{4n + 1}$. Let $SZ(4n, \mathbb{R})$ denote the center subgroup of unimodular~\cite{Rotman} (See page 20) scalar matrix representations of $SL(4n, \mathbb{R})$. The unit hypersphere \(S^{4n - 1} \subseteq S^{4n}\) has isometry group actions from \(SO(4n, \mathbb{R}) \subseteq SL(4n, \mathbb{R})\) (where the origin remains fixed) and where up to isomorphism, 
$$
PSL(4n, \mathbb{R}) = SL(4n, \mathbb{R})/SZ(4n, \mathbb{R}).
$$
\indent For an abelian chain complex $\oplus_{4n - 1}$$K_{4n - 1}$ the submanifold $S^{4n - 1}$ has homology group \(H_{l}(S^{4n - 1}, \mathbb{Z}) = \mathbb{Z}\) for \(l = 0\) or $4n - 1$, $\mathbb{Z}_{2}$ if \(0 < l < 4n - 1\) for odd $l$ and $0$ otherwise~\cite{May},~\cite{Madsen}. The manifold $\mathbb{C}\mathbb{P}^{2n}$ has deRham cohomology group

\[ H^{l}(\mathbb{C}\mathbb{P}^{2n}) = \left\{ \begin{array}{ll}
                                                \mathbb{R}^{1} & \mbox{if \(l = 2j \ni 0 \leq l/2 \leq 2n\),} \\
                                                 0             & \mbox{otherwise.}
                                              \end{array}
                                      \right\} \]
The Hopf manifold $S^{4n + 1}/S^{1}$ is homeomorphic to $\mathbb{C}\mathbb{P}^{2n}$~\cite{Madsen} (Chapter 14). A smooth complex projective manifold, in particular when it is simply--connected for some universal covering map, allows on it the construction of an exterior algebra by which quantum gravity theorists design various mathematical models of nature. One example of a model is the twistor space on the complex projective manifold $\mathbb{C}\mathbb{P}^{3}$~\cite{Wald} (See Chapter 14, Section 14.1). 

\indent Let \(n = 1\) or $2$. For \(n = 1\) and using Eqtns. (20)--(25), \(V_{4} =\)
\begin{eqnarray}
\zeta(2)&=&\frac{\pi^{2}r_{1}^{4}}{1!2!C_{1}} = \frac{2|B_{2}|\pi^{2}}{(1!)(2!)(1)}\nonumber \\
        &=&\frac{\pi^{2}}{6} = 1.6449\ldots \nonumber
\end{eqnarray}
where \(B_{2} = 1/6\), \(r_{1} = \sqrt[4]{2|B_{2}|}\). This is the same result we obtained in Section 3 through other means (i. e., by Proposition 1). From the topological equivalence between $\partial(D)$ and $S^{4n - 1}$, where $\partial(D)$ is the boundary of $S^{4n}$,
\begin{eqnarray}
\partial(D)&\cong&S^{3}\subseteq S^{4}\nonumber \\
           &\cong&\mathbb{H}\mathbb{P}^{1},\nonumber
\end{eqnarray}
where $\mathbb{H}$ is the normed, noncommutative division algebra of quaternions. For \(n = 2\),           
\begin{eqnarray}
V_{8}&=&\zeta(4)\nonumber \\
     &=&\frac{\pi^{4}r_{2}^{8}}{2!3!C_{2}},\nonumber
\end{eqnarray}
and
\begin{eqnarray}
\partial(D)&\cong&S^{7} \subseteq  S^{8}\nonumber \\
           &\cong&\mathbb{O}\mathbb{P}^{1},\nonumber
\end{eqnarray}
\mbox{where $\mathbb{O}$ is the normed, noncommutative and nonassociative division algebra of octonions}~\cite{Baez},~\cite{Crowell et Betts}. Let \(a \in S^{4}, b \in S^{8}\), with $U_{a}$, $U_{b}$ the corresponding open neighborhoods. The corresponding Hopf bundle maps are~\cite{Madsen}
\begin{eqnarray}
S^{3}&\mapsto S^{7} \mapsto S^{4},\nonumber \\
S^{7}&\mapsto S^{15} \mapsto S^{8},\nonumber
\end{eqnarray}
with two local trivializations
\begin{eqnarray}
h_{1}:U_{a} \times S^{3}&\rightarrow \pi_{a}^{-1}(U_{a}),\nonumber \\
h_{2}:U_{b} \times S^{7}&\rightarrow \pi_{b}^{-1}(U_{b}), \nonumber
\end{eqnarray}
for the two bundle maps
\begin{eqnarray}
\pi_{a}:S^{7}&\rightarrow S^{4},\nonumber \\
\pi_{b}:S^{15}&\rightarrow S^{8}.\nonumber
\end{eqnarray}
The two composition maps defined from the lifts are \(\pi_{a} \circ h_{1}\) and \(\pi_{b} \circ h_{2}\). 

\indent The following can be regarded as being a corollary to Proposition 3, namely that when \(m = 4n\), \(s = 2n\) and when $r_{n}$ is as defined in Proposition 3,
\begin{eqnarray}
{2n \choose n}&=&\frac{\pi^{2n}r_{n}^{4n}}{(n!)^{2}\zeta(2n)},\nonumber \\
\zeta(2n)     &=&\frac{\pi^{2n}r_{n}^{4n}}{(n!)^{2}{2n \choose n}},\nonumber
\end{eqnarray}
These two formulas derive immediately from Eq. (25), since
\begin{eqnarray}
C_{n}         &=&\frac{\pi^{2n}r_{n}^{4n}}{n!(n + 1)!\zeta(2n)},\nonumber \\
{2n \choose n}&=&\frac{\pi^{2n}(n + 1)r_{n}^{4n}}{n!(n + 1)!\zeta(2n)},\nonumber \\
\end{eqnarray}
From which follow the two results. We have justified our claim that there exists a combinatorial relationship between $\zeta(s)$, $C_{n}$ and ${2n \choose n}$. As these results show, the combinatorial relationship which exists through finding the content $V_{4n}$ of a $4n$--dimensional hypersphere $D$ with the radius $r_{n}$ and where \(\partial(D) = S^{4n - 1}\) (through topological equivalence), exists independently of our infinite series of ratios
$$
\left(\frac{C_{n}}{{2n \choose n}}\right)^{s},
$$
on the complex plane.
\noindent
\begin{proposition}
\emph{In Euclidean space of even dimension \(m = 4n\), \(n = 1, 2, 3\ldots\)}
$$
V_{4n} = \prod_{p}\left(1 - \frac{1}{p^{2n}}\right)^{-1}.
$$
\label{prop4}
\end{proposition}
\begin{proof}
This follows directly from the fact that
$$
\zeta(2n) = \prod_{p}\left(1 - \frac{1}{p^{2n}}\right)^{-1},
$$
and from the results derived in the proof to Proposition 3, including Eq. (20), Eq. (24) and from the fact that \(\zeta(2n) = V_{4n}\) when \(m = 4n\), \(r_{n} = \sqrt[4n]{2^{2n - 1}|B_{2n}|}\).
\end{proof}
We now have these five results, namely:
\begin{eqnarray}
\zeta(s)      &=&\sum_{n = 1}^{\infty}\frac{1}{n^{s}} \nonumber \\
              &=&\left(\frac{C_{0}}{1}\right)^{s} + \sum_{n = 1}^{\infty}\left(\frac{C_{n}}{{2n \choose n}}\right)^{s},\nonumber \\
              &=&\prod_{p}\frac{1}{\left(1 - \frac{1}{p^{s}}\right)}, \nonumber \\
C_{n}         &=&\frac{\pi^{2n}r_{n}^{4n}}{n!(n + 1)!\zeta(2n)},\nonumber \\
V_{4n}        &=&\frac{\pi^{2n}r_{n}^{4n}}{n!(n + 1)!C_{n}} = \zeta(2n).\nonumber \\
{2n \choose n}&=&\frac{\pi^{2n}r_{n}^{4n}}{(n!)^{2}\zeta(2n)},\nonumber \\
\zeta(2n)     &=&\frac{\pi^{2n}r_{n}^{4n}}{(n!)^{2}{2n \choose n}}.
\end{eqnarray}

\indent How do these results relate to the Riemann hypothesis? At first glance one might think that the values \(s = 2n = m/2\) bear little or no relation to values on the vertical line \(\mathcal{R}e(s) = 1/2\) on $\mathbb{C}$,  where \(m \equiv 0 mod 4\) is an even Euclidean dimension for the hypersphere of content $V_{4n}$ and radius $r_{n}$. We show this is not the case. Let
$$
z_{j} \in \mathbb{C}, j = 1, 2, \ldots, 12, \: a_{1}, a_{2}, a_{3}, a_{4} \in \mathbb{C},
$$
\begin{eqnarray}
z_{1}z_{4} - z_{3}z_{2} \not = 0,   z_{5}z_{8} - z_{7}z_{6}  &\not = 0, \nonumber \\
z_{9}z_{12} - z_{11}z_{10} \not = 0, a_{1}a_{4} - a_{3}a_{2} &\not= 0.\nonumber
\end{eqnarray}
Then we have conformal transformations for three constant complex numbers $s_{0}, s_{1}, s_{2}$, in $\mathcal{C}$,
\begin{eqnarray}
s_{0}&\mapsto \frac{z_{1}s_{0} + z_{2}}{z_{3}s_{0} + z_{4}} = w_{0},\nonumber \\
s_{1}&\mapsto \frac{z_{5}s_{1} + z_{6}}{z_{7}s_{1} + z_{8}} = w_{1},\nonumber \\
s_{2}&\mapsto \frac{z_{9}s_{2} + z_{10}}{z_{11}s_{2} + z_{12}} = w_{2}.\nonumber \\
\end{eqnarray}
This means the points $s_{0} = 2n$, $s_{1} = 0$ and $s_{2} = i$ in the upper complex half plane can be mapped to the three corresponding points \(w_{0} = 1/2 + 2ni, w_{1} = 1/2 + 0i\) and to some third point \(w_{2} \in \mathbb{C}\) all inside the vertical strip \(0 \leq \mathcal{R}e(s) \leq 1\). For all \(|s| > 1\), restrict the values of $\zeta(s)$ to the first Riemann sheet. Then $\zeta(s)$ is both single--valued and analytic for all \(|s| > 1\). Then we can establish a fourth conformal mapping~\cite{Dettman} (See Theorem 6.4.1, Corollary 6,4,1 in Chapter 6, Section 6.4, pages 253--254) 
$$
\zeta(2n) \mapsto \zeta\left(\frac{1}{2} + 2ni\right),
$$
by the M\"{o}bius transformation
\begin{eqnarray}
\zeta\left(\frac{1}{2} + 2ni\right)&=&\zeta\left(\frac{2nz_{1} + z_{2}}{2nz_{3} + z_{4}}\right) \nonumber \\
                                   &=&\frac{a_{1}\zeta(2n) + a_{2}}{a_{3}\zeta(2n) + a_{4}},\nonumber \\
\end{eqnarray}
$$
z_{1}z_{4} - z_{3}z_{2} \not = 0, \: a_{1}a_{4} - a_{3}a_{2} \not = 0,
$$
where the complex numbers $1/2 + 2ni$, \(n = 1, 2, 3\ldots\), may or may not be zeroes of $\zeta(s)$. However at the very least if these numbers do not appear as zeroes, the map
$$
2n \mapsto \frac{1}{2} + 2ni, \: n = 1, 2\ldots,
$$
indicates where on the vertical line \(\mathcal{R}e(s) = 1/2\) some actual root $1/2 + t_{0}i$ (such that \(\zeta(1/2 + t_{0}i) = 0\)) might be, meaning whenever \(\mathcal{I}m(s) = t_{0} \approx 2n\).
\section{Catalan Numbers, $\zeta(s)$ and their Relevance to certain Problems in Physics}
\subsection{The Weil Zeta Function and Quantum Chaos}
We have proved relationships exist between the numbers $\zeta(s), C_{n}, {2n \choose n}$ and $V_{4n}$, in particular when \(s = 2n\) and when \(m \equiv 0 mod 4\) is an even Euclidean dimension. Trivial zeroes of $\zeta(2n)$ occur for integer \(n < 0\), since for this case $\zeta(s)$ can be expressed in terms that contain odd Bernoulli numbers~\cite{Edwards} (Chapter 1). In addition as we shall show in Section 5, if the Riemann Hypothesis is true, then roots $1/2 + ti$ such that \(\zeta(1/2 + ti) = 0\) within the vertical strip \(0 \leq \mathcal{R}e(s) \leq 1\) have a definite impact on values for the zeroes to the infinite series in Proposition 1. A correct proof of the Riemann Hypothesis (RH) would indicate not only that the nontrivial zeroes of the traditional Riemann zeta function 
$$
\zeta(s) = \sum_{n = 1}^{\infty}\frac{1}{n^{s}}
$$ 
lie on the line $Re(s) = 1/2$ in the complex plane, but also would relate to the actual distribution of the primes~\cite{Ireland} (Proposition 16.1.1, Chapter 16, Section 1). In fact Ruelle~\cite{Ruelle} has described $\zeta(s)$ as being a generating function for the prime numbers when it is identified with the Euler product. For nonnegative integer $m$ and for an iterative dynamical map $f$ in some Hamiltonian physical system, there exists a \emph{Weil zeta function}
$$
\zeta_{W}(Z) := exp \sum_{m = 1}^{\infty}\frac{Z^{m}}{m}|Fixf^{m}|,
$$
where $|Fixf^{m}|$ is the number of fixed points (i. e., points of stability) for the iterative map in the phase space $M$. Here the iterative map $f$ for a dynamical system $(M, f)$ actually is a Frobenius map $Z \mapsto Z^{q}$, where $q$ is the order of finite field $\mathbb{F}_{q}$. This latter field is the algebraic closure for an algebraic variety the extension of which is $\mathbb{F}_{q}$ to obtain the space $M$~\cite{Ruelle}. For smooth ($C^\infty$) manifolds $M$ the map $f$ can be a diffeomorphism. A. Terras has discussed how quantum theorists have found there is some relationship between the Riemann zeta function, the distribution of prime numbers and the \emph{level spacings} \(|a_{k} - a_{k - 1}|\) for certain histograms studied in quantum chaos~\cite{Terras},~\cite{Ruelle} (page 889). A level spacing is the absolute value of the difference between each two successive energy levels appearing in the histogram for real numbers $a_{k - 1}$, $a_{k}$ found from the discrete energy spectrum (eigenvalues) for the random matrices in the two unitary groups $U(n)$, $SU(n)$~\cite{Kowalski},~\cite{Mehta},~\cite{Schwabl},~\cite{Terras}. 

\indent These random matrices~\cite{Fulman} are square matrices with elements that are Gaussian distributed random variables. This indicates their relevance in quantum chaos~\cite{Mehta}. A. Terras has compared the level spacings in the energy spectrum for the heavy nucleus for Erbium (element number sixty-eight in the Periodic Table of the Elements) with complex zeroes numbers $1551$ to $1600$ for $\zeta(s)$ (~\cite{Terras}. Compare columns (c) and (e) in Figure 2, page 124). Physicist Freeman Dyson has done considerable research on random matrices [17]. The Riemann zeta function also has relevance in this field of physics~\cite{Peterson}. In fact there is good reason to believe that, if the Riemann Hypothesis is true (the current status of research on the RH can be found at: \texttt{www.claymath.org}) , then the eigenvalues of random matrices bear some important relationship to the successive gaps between the zeroes of $\zeta(s)$ (~\cite{Terras}, Section 6.5 and Table IV. See also Section 6.6) on the line 
\(\mathcal{R}e(s) = 1/2\). It stands to reason that if ever the two successive zeroes $1/2 + it_{k}$ and $1/2 + it_{k + 1}$, are such that $t_{k}$ and \(t_{k + 1} > t_{k}\) are successive prime numbers, that
$$
\left(\frac{1}{2} + it_{k + 1}\right) - \left(\frac{1}{2} + it_{k}\right)
$$
$$
= i\Delta \! t_{k},
$$
where \(\Delta \! t_{k} - 1 = t_{k + 1} - t_{k} - 1\) is a prime gap~\cite{Pomerance},~\cite{Mozzochi},~\cite{Nicely},~\cite{Betts}. 

\indent If random matrices do give information on the location of gaps between the successive zeroes of $\zeta(s)$, then they also give information on the distribution of prime gaps on the line \(\mathcal{R}e(s) = 1/2\), since we have
\begin{equation}
\zeta(s) = \prod_{p}\frac{1}{\left(1 - \frac{1}{p^{s}}\right)}.
\end{equation}
Then the numbers $\zeta(2n)$, $V_{4n}$ and $C_{n}$ help also to give information about these same primes, due to Eqtns. (25)--(26) and by Proposition 4 (See Section 3.1). 
\subsection{The Casimir Effect and Catalan number $C_{3}$}
The Casimir effect involves the presence of vacuum state energy between two small metal plates in the laboratory separated by a small distance $d$. Such energy in fact has been detected and measured in the laboratory. One computational result with the expected value $<E>$ for energy is
\begin{equation}
\frac{<E>}{A} = - \frac{\hbar c \pi^{2}}{6d^{2}}\zeta(-3),
\end{equation}
where $\hbar$ is Planck's constant divided by $2\pi$, $c$ is the speed of light and $A$ is the area of the plates. Since one knows that
\begin{equation}
\zeta(-n) = -\frac{B_{n + 1}}{n + 1},
\end{equation}
and since
\begin{equation}
\frac{C_{n}}{{2n \choose n}} = \frac{1}{n + 1},
\end{equation}
we obtain at once for Eqtn. (32),
\begin{equation}
\frac{<E>}{A} = \frac{\hbar c \pi^{2}}{6d^{2}}\frac{B_{4}C_{3}}{{6 \choose 3}}.
\end{equation}
\section{Summary}
We summarize briefly our results:
\begin{enumerate}
\item The Riemann zeta function has a combinatorial description (See also the remarks at the end of Section 3.1), namely:
\begin{equation}
\zeta(s) = \left(\frac{C_{0}}{1}\right)^{s} + \sum_{n = 1}^{\infty}\left(\frac{C_{n}}{{2n \choose n}}\right)^{s},
\end{equation}
given as an infinite sum of ratios made from Catalan numbers $C_{n}$ and central binomial coefficients ${2n \choose n}$. This infinite series, in its own right, is a uniformly convergent series of analytic functions on $\mathbb{C}$.
\item There exist two closed formulas, namely
\begin{eqnarray}
\zeta(2n)&=&\frac{\pi^{2n}r_{n}^{4n}}{n!(n + 1)!C_{n}},\\
C_{n}    &=&\frac{\pi^{2n}r_{n}^{4n}}{n!(n + 1)!\zeta(2n)}, 
\end{eqnarray}
by which one can compute $\zeta(2n)$ by using $C_{n}$ and compute $C_{n}$ by using $\zeta(2n)$, where $\zeta(2n)$ is actually equal to the content $V_{m}$ of a hypersphere of radius $r_{n} = \sqrt[4n]{2^{2n - 1}|B_{2n}|}$ imbedded in Euclidean space of even dimension \(m = 4n\), where $B_{2n}$ is the $2n^{th}$ Bernoulli number.
\item The intrinsic properties of the four numbers $\zeta(s)$, $C_{n}$, ${2n \choose n}$ and $V_{4n}$, relate them together when \(s = 2n\), so that $\zeta(2n)$ has a combinatorial relationship, namely: 
\begin{eqnarray}
{2n \choose n}&=&\frac{\pi^{2n}r_{n}^{4n}}{(n!)^{2}\zeta(2n)},\\
\zeta(2n)     &=&\frac{\pi^{2n}r_{n}^{4n}}{(n!)^{2}{2n \choose n}},
\end{eqnarray}
when \(m = 4n\) is the Euclidean dimension for a hypersphere with radius $r_{n}$ and content \(V_{4n} = \zeta(2n)\), as we found at the end of Section 3.1. So when \(s = 2n\) we also obtain
\begin{equation}
V_{4n} = \prod_{p}\frac{1}{1 - \frac{1}{p^{s}}},
\end{equation}
where $V_{4n}$ is the content of a $4n$-dimensional hypersphere of radius \(r_{4n} = [2^{2n - 1}|B_{2n}|]^{1/4n}\), which is homeomorphic to the unit $4n$-dimensional hypersphere.\\
\item If the RH is true, the distribution of primes, related to 
\begin{equation}
\zeta(s) = \prod_{p}\frac{1}{1 - \frac{1}{p^{s}}},
\end{equation}
has some kind of combinatorial relationship, since
\begin{eqnarray}
 &\left(\frac{C_{0}}{1}\right)^{s} + \sum_{n = 1}^{\infty}\left(\frac{C_{n}}{{2n \choose n}}\right)^{s} = \prod_{p}\frac{1}{1 - \frac{1}{p^{s}}}\nonumber \\
 &\Longrightarrow \sum_{n = 1}^{\infty}\left(\frac{C_{n}}{{2n \choose n}}\right)^{s} = \prod_{p}\frac{1}{1 - \frac{1}{p^{s}}} - 1,\nonumber \\
 &C_{0} = 1.
\end{eqnarray}
\end{enumerate}

\pagebreak
\begin{center}
ACKNOWLEDGEMENTS
\end{center}
\noindent I wish to thank my ``de facto" mathematical physics mentor (from 2002-2005) Doctor L. Crowell (University of Washington Physics Email List, 2001-2004), for some suggestions on improving the discussion in Section 1. I also wish to thank physicists Gerald Fitzpatrick and Bob Zannelli on the List, for their illuminating physics list discussions that occasionally touched on the Casimir effect.

\pagebreak

\end{document}